\documentclass[11pt]{article}
\usepackage{graphicx}
\usepackage{amssymb, amsthm, amsmath, amsfonts}
\usepackage{mathrsfs,microtype}
\usepackage{epstopdf,color}
\usepackage{tikz}
\usepackage{tkz-berge}
\usetikzlibrary{positioning}
\usepackage{hyperref}
\usepackage{verbatim}

\usepackage{mathtools}

\DeclarePairedDelimiter\abs{\lvert}{\rvert}

\usetikzlibrary{arrows,decorations.pathmorphing,decorations.footprints,fadings,calc,trees,mindmap,shadows,decorations.text,patterns,positioning,shapes,matrix,fit,through,decorations.pathreplacing}
\usepackage[margin=1in]{geometry}

\theoremstyle{plain}
\newtheorem{theorem}{Theorem}[section]

\newtheorem{conjecture}[theorem]{Conjecture}

\theoremstyle{definition}
\newtheorem{definition}[theorem]{Definition}

\theoremstyle{remark}

\title{Independent sets in $n$-vertex $k$-chromatic $\ell$-connected graphs}
\author{John Engbers\thanks{john.engbers@marquette.edu; Department of Mathematical and Statistical Sciences, Marquette University, Milwaukee, WI 53201. Research supported by the Simons Foundation grant 524418.} \and Lauren Keough\thanks{keoulaur@gvsu.edu; Department of Mathematics, Grand Valley State University, Allendale, MI 49401.} \and Taylor Short\thanks{shorttay@gvsu.edu; Department of Mathematics, Grand Valley State University, Allendale, MI 49401.}}
\date{\today }

\begin{document}

\maketitle

\begin{abstract}
We study the problem of maximizing the number of independent sets in $n$-vertex $k$-chromatic $\ell$-connected graphs. First we consider maximizing the total number of independent sets in such graphs with $n$ sufficiently large, and for this problem we use a stability argument to find the unique extremal graph. We show that our result holds within the larger family of $n$-vertex $k$-chromatic graphs with minimum degree at least $\ell$, again for $n$ sufficiently large. 
We also maximize the number of independent sets of each fixed size in $n$-vertex 3-chromatic 2-connected graphs. We finally address maximizing the number of independent sets of size 2 (equivalently, minimizing the number of edges) over all $n$-vertex $k$-chromatic $\ell$-connected graphs.
\end{abstract}

%=============================================================================
\section{Introduction and Statement of Results}
%=============================================================================

Given a finite, simple graph $G = (V(G),E(G))$, an independent set $I$ is a subset of $V(G)$ so that if $v,w \in I$, then $vw \notin E(G)$.  The \emph{size} of an independent set is $|I|$.  We will let $i(G)$ denote the number of independent sets in $G$ and $i_t(G)$ denote the number of independent sets of size $t$ in $G$.  The quantity $i(G)$ has also been called the \emph{Fibonacci number} of $G$ \cite{ProdingerTichy} and the \emph{Merrifield-Simmons index} of $G$ \cite{MerrifieldSimmons}.

There has been a large number of papers devoted to finding the maximum and minimum values of $i(G)$ and $i_t(G)$ as $G$ ranges over some family of graphs.  For a sampling of these results, we refer the reader to two surveys \cite{Cutler,Zhao} and the references found therein.

A \emph{proper vertex coloring} of a graph $G$ is an assignment of a color to each vertex so that no edge is monochromatic.  A graph $G$ is \emph{$k$-chromatic} if there exists a proper coloring using $k$ colors but not one with $k-1$ colors. We call $G$ \emph{$k$-critical} if it is $k$-chromatic and every proper subgraph of $G$ is at most $(k-1)$-chromatic. Finally, a graph $G$ is \emph{$\ell$-connected} if $|V(G)|>\ell$ and any graph obtained by deleting fewer than $\ell$ vertices is connected. Recently Fox, He, and Manners \cite{FoxHeManners} proved an old conjecture of Tomescu by finding the $n$-vertex $k$-chromatic connected graph with the maximum number of proper vertex colorings that uses $k$ colors.

This focus of this note is on maximizing $i(G)$ and $i_t(G)$ within the family of $n$-vertex $k$-chromatic $\ell$-connected graphs.  When $\ell=0$ and $\ell=1$, the maximum number of independent sets, and independent sets of each fixed size $t$, in these families was determined in \cite{EngbersErey}.  Our first result generalizes this to maximizing $i(G)$ when $\ell>1$ for $n$ large.  Before we state it, we first define the extremal graphs for the various values of $k$ and $\ell$. Recall that for graphs $G_1$ and $G_2$, the graph $G_1 \vee G_2$ has vertex set $V(G_1) \sqcup V(G_2)$ and edge set $E(G_1) \cup E(G_2) \cup \{xy:x \in V(G_1),y \in V(G_2)\}$. We denote the complete and empty graphs on $n$ vertices by $K_n$ and $E_n$, respectively. 

\begin{definition}
Fix $k \geq 2$ and $\ell \geq 1$.  For $k \leq \ell$, let $G^*:= (K_{k-1} \cup E_{\ell-k+1}) \vee E_{n-\ell}$, and for $k > \ell$ let $G^* := K_\ell \vee (K_{k-\ell} \cup E_{n-k})$. See Figure \ref{fig-G^*}.
\end{definition}

\begin{figure}[ht!]
    \centering
\begin{tikzpicture}[scale=.7]
		\node at (-1,3.75) {$G^*$ ($k \leq \ell$)};
		\node (v1) at (1,1) [circle,draw] {\fontsize{7}{5.2}\selectfont {$k$-1}};
		\node (v2) at (1,2) [circle,draw,scale=.4,fill] {};
		\draw[densely dotted] (.35,.35) -- (.35,2.35);
		\draw[densely dotted] (.35,2.35) -- (1.65,2.35);
		\draw[densely dotted] (1.65,2.35) -- (1.65,.35);
		\draw[densely dotted] (1.65,.35) -- (.35,.35);
		\node (v3) at (3,0) [circle,draw,scale=.4,fill] {};
		\node (v4) at (3,1) [circle,draw,scale=.4,fill] {};
		\node (v5) at (3,2) [circle,draw,scale=.4,fill] {};
		\node (v6) at (3,3) [circle,draw,scale=.4,fill] {};
		\draw[densely dotted] (2.65,-.35) -- (3.35,-.35);
		\draw[densely dotted] (3.35,-.35) -- (3.35,3.35);
		\draw[densely dotted] (3.35,3.35) -- (2.65,3.35);
		\draw[densely dotted] (2.65,3.35) -- (2.65,-.35);
		\node at (3,3.75) {$n-\ell$};
		\node at (1,2.75) {$\ell$};

		\foreach \from/\to in {v1/v3,v1/v4,v1/v5,v1/v6,v2/v3,v2/v4,v2/v5,v2/v6}
		\draw (\from) -- (\to);
	\end{tikzpicture}
\qquad\qquad\qquad %OR \qquad
\begin{tikzpicture}[scale=.8]
		\node at (0,3.75) {$G^*$ ($k>\ell$)};
			\node (v1) at (1,.8) [circle,draw,scale=.4,fill] {};
		\node (v2) at (2,2) [circle,draw,scale=.4,fill] {};
		\node at (2,2.6) {$\ell$};
		\draw[densely dotted] (1.7,2.3) -- (2.3,2.3);
		\draw[densely dotted] (2.3,2.3) -- (2.3,.7);
		\draw[densely dotted] (2.3,.7) -- (1.7,.7);
		\draw[densely dotted] (1.7,.7) -- (1.7,2.3);
		\node (v3) at (1,2.2) [circle,draw,scale=.4,fill] {};
		\draw[densely dotted] (-.1,2.5) -- (1.3,2.5);
		\draw[densely dotted] (1.3,2.5) -- (1.3,.5);
		\draw[densely dotted] (1.3,.5) -- (-.1,.5);
		\draw[densely dotted] (-.1,.5) -- (-.1,2.5);
		\node at (.6,2.8) {$k-\ell$};
		\node (v4) at (2,1) [circle,draw,scale=.4,fill] {};
		\node (v5) at (4,2) [circle,draw,scale=.4,fill] {};
		\node (v6) at (4,.5) [circle,draw,scale=.4,fill] {};
		\node (v7) at (4,1.25) [circle,draw,scale=.4,fill] {};
		\node (v8) at (4,2.75) [circle,draw,scale=.4,fill] {};
		\draw[densely dotted] (3.7,.2) -- (4.3,.2);
		\draw[densely dotted] (4.3,.2) -- (4.3,3.05);
		\draw[densely dotted] (4.3,3.05) -- (3.7,3.05);
		\draw[densely dotted] (3.7,3.05) -- (3.7,.2);
		\node at (4,3.35) {$n-k$};
		\node (v9) at (0.2,1.5) [circle,draw,scale=.4,fill] {};

		\foreach \from/\to in {v9/v1,v9/v2,v9/v3,v9/v4,v1/v3,v1/v2,v2/v3,v2/v4,v4/v3,v1/v4,v4/v5,v4/v6,v4/v7,v4/v5,v2/v5,v2/v6,v2/v7,v2/v8,v4/v8}
		\draw (\from) -- (\to);
	\end{tikzpicture}

    \caption{The graph $G^*$ for the two possibilities for $k$ and $\ell$.}
    \label{fig-G^*}
\end{figure}
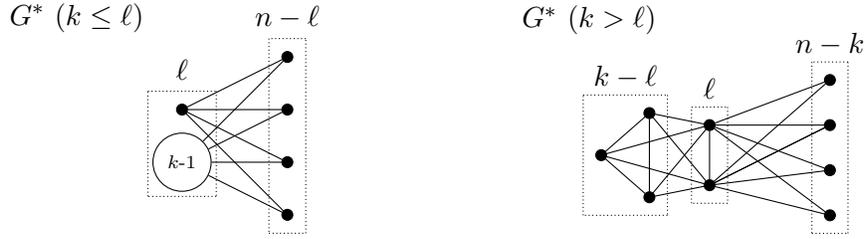
So, for example, when $k=2\leq \ell$  we have $G^*=K_{\ell,n-\ell}$, and the fact that $i(G) \leq i(G^*)$ for all $n$-vertex $\ell$-connected bipartite graphs $G$, with equality if and only if $G=G^*$, appears as Corollary 2.2 in \cite{AlexanderCutlerMink} (recall that an $\ell$-connected graph has minimum degree at least $\ell$). When $\ell=1$ and $k > 1$ the graph $G^*$ is formed from a $k$-clique with $n-k$ pendants attached to one vertex in the clique, and the fact that $i(G) \leq i(G^*)$ for all connected $k$-chromatic graphs $G$, with equality if and only if $G=G^*$, appears as Corollary 3 in \cite{EngbersErey}. (Viewing $G^*$ naturally as $K_k \cup E_{n-k}$ in the case where $\ell=0$, the analogous result appears as Corollary 2 in \cite{EngbersErey}.)

We show that this result is in fact true for all $k>2$, $\ell \geq 1$, and $n$ large.

\begin{theorem}\label{thm-i(G)}
Let $k>2$ and $\ell \geq 1$ be fixed.  
If $n > 2(k+\ell+2)\binom{6(k+\ell)}{\ell}$ and $G$ is an $n$-vertex $k$-chromatic $\ell$-connected graph, then
\[
i(G) \leq i(G^*)=\begin{cases} 
      2^{n-\ell}+k2^{\ell-k+1} -1, & \quad k \leq \ell \\
      (k-\ell+1)2^{n-k}+\ell, & \quad k > \ell 
  \end{cases},
\] 
with equality if and only if $G=G^*$.
\end{theorem}

In fact, we prove the following more general result, from which Theorem \ref{thm-i(G)} follows as $\ell$-connected graphs have minimum degree at least $\ell$.

\begin{theorem}\label{thm-asympresult}
Let $k>2$ and $\ell \geq 1$ be fixed.  
If $n > 2(k+\ell+2)\binom{6(k+\ell)}{\ell}$ and $G$ is an $n$-vertex $k$-chromatic graph with minimum degree at least $\ell$, then
\[
i(G) \leq i(G^*)=\begin{cases} 
      2^{n-\ell}+k2^{\ell-k+1} -1, & \quad k \leq \ell \\
      (k-\ell+1)2^{n-k}+\ell, & \quad k > \ell 
  \end{cases},
\]
with equality if and only if $G=G^*$.
\end{theorem}

The proof of Theorem \ref{thm-asympresult} uses a stability technique, and proceeds by showing that any graph $G$ satisfying $i(G) \geq i(G^*)$ must have a similar structure to $G^*$ in that $G$ must have a large complete bipartite subgraph $K_{\ell,cn}$ for some constant $c$.  It then breaks into two cases depending on the values of $k$ and $\ell$, where the count of $i(G)$ is driven by those independent sets that completely avoid the size $\ell$ part of $K_{\ell,cn}$.

\medskip

In \cite{EngbersErey}, the authors ask what can be said in the family of $n$-vertex $k$-chromatic $\ell$-connected graphs for independent sets of size $t$ when $\ell>1$.  We also provide some results here for specific values of $k$, $\ell$, and $t$. The first new case is that of $2$-connected $3$-chromatic graphs; we remark that results on maximizing $i(G)$ over all $n$-vertex $2$-connected graphs appear in \cite{HuaZhang}.

\begin{definition}
A \emph{theta graph} joins vertices $v$ and $w$ with three internally disjoint paths of (edge) lengths $a$, $b$, and $c$. We denote this graph by $\theta_{a,b,c}$.
\end{definition} 
We have, for example, that $\theta_{2,2,2} = K_{2,3}$, and it is apparent that $|V(\theta_{a,b,c})| = a + b + c-1$.  Note that when $a$ is even, the corresponding $vw$ path has an odd number of internal vertices. We now state the theorem.

\begin{theorem}\label{thm-2con3chrom}
Let $n \geq 4$, and let $G$ be an $n$-vertex $3$-chromatic $2$-connected graph.  Then we have the following:
\begin{itemize}
	\item if $n$ is odd, then $i_2(G) \leq i_2(C_n)$ with equality if and only if $G=C_n$;
    \item if $n$ is even, then $i_2(G) \leq i_2(\theta_{a,b,c})$, where at least one of $a$, $b$, or $c$ is even, with equality if and only if $G = \theta_{a,b,c}$ where at least one of $a$, $b$, or $c$ is even; and
	\item for all $3 \leq t \leq n-2$, $i_t(G) \leq i_t(K_2 \vee E_{n-2})$, and for $n \geq 5$ we have equality if and only if $G=K_2 \vee E_{n-2}$.
\end{itemize}
\end{theorem}

The results for maximizing $i(G)$ amongst $3$-chromatic $2$-connected graphs are consequences of \cite{HuaZhang}, and we briefly discuss this at the beginning of Section \ref{sec-2con3chrom}.  Also, as a fairly routine consequence of results for independent sets of size $t$ in $n$-vertex $\ell$-connected graphs, we show the following. 

\begin{theorem}\label{thm-bigt}
Let $k \geq 3$, $\ell \geq k$, and $n \geq 2\ell$ be fixed.  If $G$ is an $n$-vertex $k$-chromatic $\ell$-connected graph and $t \geq \ell$, then
\[
i_t(G) \leq i_t(G^*).
\]
\end{theorem}

\medskip

Finally, we also consider the problem of maximizing the number of independent sets of size $t=2$ in $k$-chromatic $\ell$-connected graphs. Note that an independent set of size 2 induces an edge in the complement of the graph, so this problem is equivalent to minimizing the number of edges. 

The problem of minimizing edges has been studied for several related families of graphs. The minimum number of edges in a $k$-chromatic graph is clearly ${k \choose 2}$. The minimum number of edges in $\ell$-connected graphs is $\lceil \frac{n\ell}{2} \rceil$ due to Harary \cite{Harary}. The minimum number of edges in $k$-critical graphs was first studied by Dirac \cite{Dirac} and Gallai \cite{Gallai-I, Gallai} and subsequently in \cite{Kostochka,Krivelevich1,Krivelevich2}. Minimizing edges in $k$-chromatic $\ell$-edge-connected graphs was considered in \cite{Westetal}, where it was briefly noted that some of the bounds also hold for $\ell$-connected graphs. 

We present two sharp bounds for the minimum number of edges in $k$-chromatic $\ell$-connected graphs for the case that $k-1>\ell >1$. The first has the extra condition that $\ell \le  n-k$ and our bound coincides exactly with the result in \cite{Westetal}, however, their proof relies on edge-connectivity. Furthermore, our techniques in the range $\ell \le n-k$ allow us to tackle the range $\ell > n-k$ as well; this appears as an unsolved case in \cite{Westetal}. All remaining cases for minimizing edges in $k$-chromatic $\ell$-connected graphs follow similarly to \cite{Westetal}, so we omit the results here.

\begin{theorem}\label{thm-minedges1}
If $G$ is a $k$-chromatic $\ell$-connected graph with $k-1>\ell>1$ and $\ell\leq n-k$ then
\[\abs{E(G)}\geq \binom{k}{2}+\frac{(n-k+1)\ell}{2}\]
and this bound is sharp.
\end{theorem}

\begin{theorem} \label{thm-minedges2}
If $G$ is a $k$-chromatic $\ell$-connected graph with $k-1>\ell>1$ and $\ell>n-k$ then
\[
|E(G)| \geq \binom{k}{2} + \binom{n-k}{2} + (n-k)(\ell-(n-k-1))
\]
and this bound is sharp.
\end{theorem}

In the rest of the paper we present the proofs of our results. We prove Theorem \ref{thm-asympresult} in Section \ref{sec-asympresult}, and we consider 2-connected, 3-chromatic graphs in Section \ref{sec-2con3chrom} where we also prove Theorem \ref{thm-bigt}. Then we consider maximizing independent sets of size $t=2$ in Section \ref{sec:t=2}. Finally, in Section \ref{sec-conclusion}, we highlight some open questions related to the results in this paper.

\section{Proof of Theorem \ref{thm-asympresult}} \label{sec-asympresult}

In this section we will prove Theorem \ref{thm-asympresult}, and to do so we will use the following results from \cite{EngbersErey}.

\begin{theorem}[\cite{EngbersErey}]\label{thm-EngbersEreyIndSets}
Let $G$ be an $n$-vertex $k$-chromatic graph.  Then
\[
i(G) \leq i(K_{k} \cup E_{n-k}) = (k+1)2^{n-k}
\]
with equality if and only if $G = K_k \cup E_{n-k}$.
\end{theorem}

\begin{theorem}[\cite{EngbersErey}]\label{thm-EngbersEreyComponents}
Let $G$ be an $n$-vertex $k$-chromatic graph with $d$ components.  Then
\[
i(G) \leq k2^{n-k} + 2^{d-1}
\]
with equality if and only if $G = (K_1 \vee (K_{k-1} \cup E_{n-k-d+1}))\cup E_{d-1}$.
\end{theorem}

We now move on to the proof.

\begin{proof}[Proof of Theorem \ref{thm-asympresult}]

Suppose that $G$ is an $n$-vertex $k$-chromatic graph with minimum degree at least $\ell$ which satisfies $i(G) \geq i(G^*)$.  We investigate the structure of the graph $G$. First note that $i(G^*) > 2^{n-k-\ell}$ holds for each $k$ and $\ell$.

\medskip

\textbf{Step 1:} {\em Show that $G$ cannot contain a large matching.}

Consider a maximum matching $M$ in $G$, and let $|M|$ denote the size of this maximum matching.  In any independent set, at most one endpoint of each edge in $M$ is in the independent set, giving three possibilities across each edge in $M$. Therefore, by only considering the restrictions on the edges in $M$, we have 
\[
i(G) \leq 3^{|M|} 2^{n-2|M|} = \left( \frac{3}{4} \right)^{|M|} 2^n.
\]
%by only considering the restrictions on the edges in $M$ (this may count some vertex subsets that are not independent, which is why this is an upper bound).

If $|M| > 3k+3\ell$, then $\left( \frac{3}{4}\right)^{|M|} < \left(\frac{3}{4}\right)^{3k+3\ell} =\left( \frac{27}{64}\right)^{k+\ell}< \left(\frac{1}{2}\right)^{k+\ell}$ and so $i(G) < 2^{n-k-\ell}$, which contradicts the assumption that $i(G) \geq i(G^*)$.  Therefore, we know that the maximum size of a matching $M$ in $G$ at most $3k+3\ell$.

\medskip

\textbf{Step 2:} {\em Show that there is a constant $c = c(k,\ell)$ so that $G$ contains $K_{\ell,cn}$ as a subgraph.}

Let $M$ be a maximum matching.  By Step 1, there are $2|M| \leq 6(k+\ell)$ vertices that are endpoints in $M$; call this set of vertices $J$.  The set $I=V(G) \setminus J$ has size $|I| \geq n-6(k+\ell)$ and, by maximality of $M$, must form an independent set.  Since $G$ has minimum degree at least $\ell$, each vertex in $I$ must have at least $\ell$ neighbors in $J$.  %So each vertex in $I$ can be assigned a subset of size $\ell$ from $J$ (if there are more than $\ell$ neighbors, select the first $\ell$ in some fixed ordering of $J$).  
The pigeonhole principle then produces some set $L$ of size $\ell$, with $L \subseteq J$, having at least $(n-|J|) / \binom{|J|}{\ell} \geq cn$ common neighbors in $I$ for some constant $c = c(k,\ell)>0$.  Using that $n-|J| \geq n/2$ when $n \geq 12(k+\ell)$ and $\binom{|J|}{\ell} \leq \binom{6(k+\ell)}{\ell}$, we see that we can take $c = 1/(2\binom{6(k+\ell)}{\ell})$. %(here, I'm thinking $n-|J| \geq n/2$ since we want $n$ to be large anyways, and $\binom{|J|}{\ell} \leq |J|^{\ell} \leq (6(k+\ell))^\ell$, so something like $c=\frac{1}{2\cdot (6(k+\ell))^\ell}$ should do?).  
This shows that $G$ contains a (not necessarily induced) subgraph $K_{\ell,cn}$.

\medskip

\textbf{Step 3:} {\em Estimate the number of independent sets in $G$ that include a vertex from $L$.} % that in $K_{\ell,cn}$ the independent sets containing a vertex from $L$ (the side with size $\ell$) are inconsequential:

%Suppose that at least one of the vertices in $L$ is in the independent set.  
There are at most $2^{\ell}$ ways to include at least one vertex from $L$. % (the exact count is $2^{\ell}-1$).  
Then none of the at least $cn$ common neighbors of $L$ can be in the independent set, %so allowing for any of the other vertices to be in the independent set 
so this gives an upper bound of  
\begin{equation}\label{eqn-vertexinL}
2^{\ell} \cdot 2^{n-cn-\ell} = \left( \frac{1}{2} \right)^{cn} 2^n
\end{equation}
independent sets that contain some vertex from $L$.  %Since $c>0$, for large enough $n$ this is SMALL relative to $2^{n-k}$ or $2^{n-\ell}$, as $k$ and $\ell$ are fixed. To find this value, it is enough to have $(1/2)^{cn} < (1/2)^{k+\ell}$, or $cn > k+\ell$, or $n > \frac{1}{c}(k+\ell)$ (recall $c$ depends on $k$ and $\ell$). But bigger $n$ than this exponentially dampen this term, so we may want $n$ to be even bigger!

\medskip

\textbf{Step 4:} {\em Find an upper bound on the number of independent sets in $G$.  }

We have a bound from those independent sets that contain a vertex from $L$ above in \eqref{eqn-vertexinL}.  Those that do not contain a vertex from $L$ correspond to the independent sets in $G'$, the graph obtained by deleting $L$.  Note that $|V(G')| = n-\ell$. The chromatic number of $G'$ must be at most $k$, and so if the chromatic number is $m \leq k$ then by Theorem \ref{thm-EngbersEreyIndSets} we have at most $(m+1)2^{n-\ell-m}$ independent sets of this type, with equality if and only if $G - L$ is a complete graph on $m$ vertices with $n-\ell-m$ isolated vertices. To compare these maximal values for various $m$, note that for $m>1$ we have
\begin{equation}\label{eqn-compare}
(m+1)2^{n-\ell-m} = \frac{m+1}{2}2^{n-\ell-m+1} < m2^{n-\ell-m+1} = ((m-1)+1) 2^{n-\ell-(m-1)}.
\end{equation}
%so this is maximized by the smallest $m$ we can have. Intuitively, the deletion of the $\ell$ vertices should lower the chromatic number as much as possible. 

We now look at the two cases depending on the values of $k$ and $\ell$.

\bigskip

\textbf{Case 1 ($k > \ell$):} Suppose that $k>\ell$. We first argue that the chromatic number of $G'$ must be $k-\ell$.  If not, then by the bound from \eqref{eqn-compare} the graph $G'$ has chromatic number at least $m=k-\ell+1$, and so at most $(k-\ell+2)2^{n-k-1}$ independent sets of this type.  Combining this with \eqref{eqn-vertexinL} gives
\[
i(G) \leq (k-\ell+2)2^{n-k-1} + \left( \frac{1}{2} \right)^{cn} 2^n = \left( \frac{k-\ell+2}{2} +  \left( \frac{1}{2} \right)^{cn}2^k \right) 2^{n-k}.
\]
For $n>(k+1)/c$ (recalling also that $\ell \geq 1$), we have $i(G) < (k-\ell+1)2^{n-k}+\ell$, which is a contradiction.  % as the inconsequential term doesn't make up for the loss in the dominant term. (Take $(1/2)^{cn}$ to be at most $2^{-k-\ell}$...that might do it...this feels like a calculus problem to solve for what $n$ needs to be here.)

We now know that the chromatic number of $G'$ is $k-\ell$, and we next aim to show that $G'$ has many components. By Theorem \ref{thm-EngbersEreyComponents} an ($n-\ell$)-vertex $k$-chromatic graph with exactly $d$ components has at most $k2^{(n-\ell)-k}+2^{d-1}$ independent sets; note that this is an increasing function of $d$. % (notice that the larger the $c$, the larger the bound).  
Therefore, if $G'$ has at most $n-k$ components, we have
\[
i(G) \leq (k-\ell) 2^{n-\ell - (k-\ell)} + 2^{n-k-1}+\left( \frac{1}{2} \right)^{cn} 2^n = (k-\ell)2^{n-k} + \left( \frac{1}{2} + \left( \frac{1}{2} \right)^{cn}2^k \right) 2^{n-k}.
\]
For $n > (k+1)/c$, we have $i(G) < (k-\ell+1)2^{n-k}+\ell$, which is a contradiction.

Therefore, we know:
\begin{itemize}
	\item the chromatic number of $G'$ is $k-\ell$;
    \item $G'$ has $n-\ell$ vertices; and
    \item $G'$ has at least $n-k+1$ components.
\end{itemize}
This is only possible if $G'$ has exactly $n-\ell - (k-\ell)+1 = n-k+1$ components, one component is $K_{k-\ell}$, and the rest are isolated vertices.  Therefore $G'$ must be the graph $K_{k-\ell} \cup E_{n-k}$.

Now, recall that $G$ has minimum degree at least $\ell$ and chromatic number $k$.  The minimum degree condition forces each vertex in $L$ to be adjacent to each isolated vertex in $G'$. 

If some vertex in $L$ is not adjacent to some vertex in the complete component of $G'$, then those two vertices can be assigned the same color in a proper coloring and this can easily be extended to a $(k-1)$-coloring of the vertices of $G$, which contradicts the assumption that $G$ is $k$-chromatic.  Therefore the vertices in $L$ must form a dominating set in $G$. Furthermore, they must all be adjacent, or similar argument shows that $G$ can be properly colored with at most $k-1$ colors.  Therefore $G$ must be the graph $G^*$.

\medskip

\textbf{Case 2 ($k \leq \ell$):} Suppose now that $k \leq \ell$. Recall that $G'$ is the graph obtained from $G$ by deleting $L$, and that in this case $i(G^*)= 2^{n-\ell} + k2^{\ell-k+1} -1$.  First suppose that $G'$ contains some edge $e$.  At most one of the endpoints of $e$ can be in an independent set, and so this combined with \eqref{eqn-vertexinL} gives
\[
i(G) \leq 3\cdot 2^{n-\ell-2} + \left( \frac{1}{2} \right)^{cn} 2^n = \left(\frac{3}{4} + \left( \frac{1}{2} \right)^{cn} 2^{\ell} \right) 2^{n-\ell}<2^{n-\ell}
\]
where the strict inequality holds for $n>(\ell+2)/c$. This is a contradiction to the assumption on $G$.

Therefore $G'$ must be the empty graph.  As $G$ has minimum degree $\ell$, this means that each vertex in $G'$ is adjacent to each vertex in $L$.  Since $G$ is $k$-chromatic, the induced graph on $L$ must be $(k-1)$-chromatic. %, which implies $|L| = \ell\geq k-1$. 
Since all edges are present between $G'$ and $L$, we have
\[
i(G) = i(L) + 2^{n-\ell}-1.
\]

Now, $L$ has $\ell$ vertices and chromatic number $k-1$, and so we know from Theorem \ref{thm-EngbersEreyIndSets} that it has at most $i(K_{k-1} \cup E_{\ell-k+1}) = k2^{\ell-k+1}$ independent sets, with equality if and only if $L=K_{k-1}\cup E_{\ell-k+1}$. Therefore
\[
i(G) = i(L) + 2^{n-\ell}-1 \leq k2^{\ell-k+1} + 2^{n-\ell} -1,
\]
with equality if and only if $L=K_{k-1}\cup E_{\ell-k+1}$, which implies that $G=G^*$.
\end{proof}

\section{$2$-connected $3$-chromatic and the Proof of Theorem \ref{thm-bigt}} \label{sec-2con3chrom}
In this section we first completely classify the $2$-connected $3$-chromatic graphs that maximize the total count of independent sets and the total count of independent sets of each (non-trivial) fixed size. 
\subsection{Total count of independent sets}
%We start with the following theorem, from which the total count mostly follows.

We first show that the result for the total number of independent sets is essentially a corollary to a result from \cite{HuaZhang}.  There it is proved that if $G$ is a $2$-connected graph with $n \geq 4$, then $i(G) \leq 2^{n-2}+3$ with equality if and only if $G$ is $K_{2,n-2}$ or $C_5$.
%\begin{theorem}[\cite{HuaZhang}]\label{thm-HZ}
%Let $G$ be a $2$-connected graph with $n \geq 4$.  Then
%\[
%i(G) \leq 2^{n-2}+3
%\]
%with equality if and only if $G$ is $K_{2,n-2}$ or $C_5$.
%\end{theorem}
Since $C_5$ is $3$-chromatic and $K_{2,n-2}$ is not, it follows that if $n=5$ we have that $C_5$ is the $2$-connected $3$-chromatic graph with the most number of independent sets.  %Let $K_2 \vee E_{n-2}$ denote the graph which begins with $K_{2,n-2}$ and joins the vertices in the size two partition class; this is a $3$-chromatic graph.  It is easy to see that 
For $n \neq 5$, note that $K_2 \vee E_{n-1}$ is $3$-chromatic and
\[
i(K_{2,n-2}) = i(K_2 \vee E_{n-2})+1,
\]
so for $n\geq 4$ and $n \neq 5$ the characterization of equality implies that $K_2 \vee E_{n-2}$ is the (not necessarily unique) graph with the maximum number of independent sets.  

\subsection{Size $t$ independent sets}
Now we move to independent sets of size $t$.  
For $t \geq 3$, we have the following results for $2$-connected graphs and graphs with fixed minimum degree $\delta \geq 3$. 

\begin{theorem}[\cite{EG}]\label{thm-EG}
Let $n \geq 4$.  For every $t \geq 3$, every $n$-vertex graph $G$ with minimum degree at least $2$ satisfies
\[
i_t(G) \leq i_t(K_{2,n-2}).
\]
For $n \geq 5$ and $3 \leq t \leq n-2$ we have equality if and only if $G$ is $H \vee E_{n-2}$, where $H$ is any graph on two vertices.
\end{theorem}

\begin{theorem}[\cite{GLS}]\label{thm-GLS}
Let $n \geq 2\delta$. For every $t \geq 3$, every $n$-vertex graph $G$ with minimum degree at least $\delta$ satisfies
\[
i_t(G) \leq i_t(K_{\delta,n-\delta})
\]
and when $3 \leq t \leq \delta$, $K_{\delta,n-\delta}$ is the unique extremal graph.
\end{theorem}

\begin{proof}[Proof of Theorem \ref{thm-2con3chrom}]
First, we analyze the case when $t=2$; here to maximize $i_2(G)$ we want to minimize $|E(G)|$. %For the cycle $C_n$ we have
%\[
%i_t(C_n) = \binom{n-t}{t} + \binom{n-t-1}{t-1}
%\]
%(see for example \cite{HopkinsStaton}).  
%If $n\geq 4$ is odd, the graph $C_n$ is $3$-chromatic and $2$-connected. % and $i_2(C_n) = \binom{n-2}{2} + (n-3) > \binom{n-2}{2} = i_2(K_{2,n-2}+e)$. 
A $2$-connected graph $G$ has minimum degree at least $2$, and so
\[
|E(G)| = \frac{1}{2}\sum_v d(v) \geq n,
\]
with equality if and only if $G$ is $2$-regular, which implies that
\[
i_2(G) \leq \binom{n}{2} - n = % \frac{n(n-3)}{2} = \binom{n-2}{2} + (n-3) = 
i_2(C_n).
\]
The equality characterization for $n$ odd follows readily as $C_n$ is the only $2$-connected $2$-regular graph.

When $n$ is even and $G$ is 2-connected, we still have the bound of $E(G) \geq n$ with equality if and only if $G$ is $2$-regular if and only if $G=C_n$.  Since in this case $C_n$ is not $3$-chromatic, this shows that 
\[
|E(G)| \geq n+1,
\]
which proves the inequality as for an $n$-vertex theta graph $\theta_{a,b,c}$ we have $|E(\theta_{a,b,c})|=n+1$.  We now work on the cases for equality.
%and so some vertex in $G$ has degree greater than $2$. This is enough to prove the inequality as for an $n$-vertex theta graph $\theta_{a,b,c}$ we have $|E(\theta_{a,b,c})| = n+1$; we now work on the cases for equality. 

Using that $G$ is not $2$-regular, by the handshaking lemma there must be at least two vertices of degree at least $3$, or one vertex of degree at least $4$. 
A single vertex of degree $4$ with the remaining vertices having degree $2$ is not possible in a $2$-connected $G$, as the deletion of the degree $4$ vertex disconnects the graph.  
So now suppose $G$ has degree $3$ vertices $v$ and $w$, and the remaining vertices of $G$ have degree $2$.  If two of the edges out of $v$ are on a cycle that misses $w$, then again the deletion of $v$ disconnects the graph, which contradicts that $G$ is $2$-connected.  So each edge out of $v$ must be on some path that ends at $w$, which implies that $G$ is a theta graph.  %Note that theta graphs are $2$-connected, but may not be $3$-chromatic.  

Given that $G$ is a theta graph $\theta_{a,b,c}$, we need the conditions on $a$, $b$, and $c$ so that $G$ is $3$-chromatic. By coloring $v$ and $w$ with different colors and then the paths between them, we see that the chromatic number is $2$ when $a$, $b$, and $c$ are all odd.  So at least one of $a$, $b$, and $c$ must be even, and by parity considerations not all of $a$, $b$, and $c$ are even, so one parameter is even and another is odd. These two paths form an odd cycle in the graph, which shows that $\theta_{a,b,c}$ is indeed $3$-chromatic.  This implies the characterization of equality. %This characterizes the graphs which are $2$-connected and $3$-chromatic which minimize $|E(G)|$, and so these maximize $i_2(G)$ with the characterization of equality evident.

Now we consider $t \geq 3$. When $n \leq 4$ there are no $n$-vertex $3$-chromatic graphs that have an independent set of size $t\geq 3$. 
Since $G$ has minimum degree at least $2$, Theorem \ref{thm-EG} implies that for every $n \geq 5$ and $t \geq 3$, we have %every $n$-vertex $2$-connected graph with $n \geq 4$ satisfies
\[
i_t(G) \leq i_t(K_{2,n-2}),
\]
with equality if and only if $G$ is $K_{2,n-2}$ or $K_2 \vee E_{n-2}$. Recalling that $K_{2,n-2}$ is bipartite, this proves the result and the characterization of equality for $n \geq 5$. 
\end{proof}

\subsection{Proof of Theorem \ref{thm-bigt}} Suppose that $n \geq 2\ell$. By Theorem \ref{thm-GLS}, $K_{\ell,n-\ell}$ is an $\ell$-connected graph that has the most number of independent sets of size $t \geq 3$.  When $t\geq \ell+1$, these independent sets consist of $t$ vertices from the size $n-\ell$ partition class.  So if $k \leq \ell$, then $i_t(G^*) = i_t(K_{\ell,n-\ell})$, and therefore $i_t(G) \leq i_t(G^*)$ for all $k$-chromatic $\ell$-connected graphs, where $k \leq \ell$, $t \geq \ell+1$, and $n \geq 2\ell$.

When $t = \ell \geq 3$, Theorem \ref{thm-GLS} gives that $K_{\ell,n-\ell}$ is the \emph{unique} $\ell$-connected graph with the most number of independent sets of size $t$.  Since for $t=\ell \geq 3$ we have $i_t(G^*) = i_t(K_{\ell,n-\ell})-1$ and for $k$-chromatic $\ell$-connected $G$ we have $i_t(G)<i_t(K_{\ell,n-\ell})$, this shows that $i_t(G) \leq i_t(G^*)$.

\section{The Case of $t=2$: Minimizing Edges}\label{sec:t=2}

In this section, we consider maximizing the number of independent sets of size $t=2$ in $k$-chromatic $\ell$-connected graphs. As previously mentioned, this problem is equivalent to minimizing the number of edges in such graphs, so our results and proofs are stated as such.

\begin{comment}
Here's a chart with some of these minimizing/maximizing edge type problems for the sake of literature search:

\begin{center}
\begin{tabular}{c|c|c}
&$k$-chromatic &$\ell$-connected\\
\hline
minimizing edges &$\binom{k}{2}$ (proof by contradiction) &$\lceil\frac{nl}{2}\rceil$ (Harary)\\
\hline
maximizing (consider critical case) &? &Ando and Egawa
\end{tabular}
\end{center}

\textcolor{blue}{LK: Did we say anything about the $\ell=1$ case? Should be a $K_k$ with a tree/forest hanging off of it} \textcolor{red}{JE: We have $\ell=1$ in \cite{EngbersErey}; it is cited there from a different paper. When $k=3$ any unicyclic graph with an odd cycle works (so there are just more extremal examples).}
\end{comment}

\subsection{Proof of Theorem \ref{thm-minedges1}}
We start by constructing a graph that achieves the bound. We make use of the $n$-vertex, $\ell$-connected Harary graph, which we denote by $H_{n, \ell}$. Recall that to construct $H_{n, \ell}$, place $n$ vertices $w_1$, $w_2$, \ldots, $w_n$ in order around a circle and join each vertex to the $\lfloor \frac{\ell}{2}\rfloor$ vertices closest to it in either direction.  In the event that $\ell$ is odd, then also join each vertex to the vertex directly opposite (or as opposite as possible when $n$ is odd). The Harary graph $H_{n, \ell}$ has $\lceil n\ell /2 \rceil$ edges, which is the minimum number of edges over all graphs with the same number of vertices and connectivity \cite{Harary}.

Consider the disjoint union of the complete graph $K_k$ and a Harary graph $H_{n-k, \ell}$. Since we are assuming $\ell < k-1$ and $\ell \le n-k$, we can choose $\ell$ vertices $v_1$, $v_2$, \ldots, $v_\ell$ from $K_k$ and $\ell$ adjacent vertices $w_1$, $w_2$, \ldots, $w_\ell$ along the circle in $H_{n-k, \ell}$ and connect these vertices via a matching, $v_iw_i$. Now starting with a terminal edge, $w_1w_2$, of the $\ell$-vertex path in $H_{n-k, \ell}$, we remove every other edge of the path. In the case where $\ell$ is odd and $n-k$ is also odd, then we use the one higher degree vertex from $H_{n-k, \ell}$ in our $\ell$-vertex path and delete both edges to the sides. Call this graph $G^*$ and let $H'_{n-k,\ell}$ denote the subgraph induced by the vertices of $H_{n-k,\ell}$. See Figure \ref{fig-}. The graph $G^*$ has \[{k \choose 2} + \left\lceil\frac{(n-k)\ell}{2}\right\rceil + \ell -\lfloor\ell /2\rfloor = \binom{k}{2} + \left\lceil \frac{(n-k+1)\ell}{2}\right\rceil \] edges and we claim this graph is $k$-chromatic and $\ell$-connected.

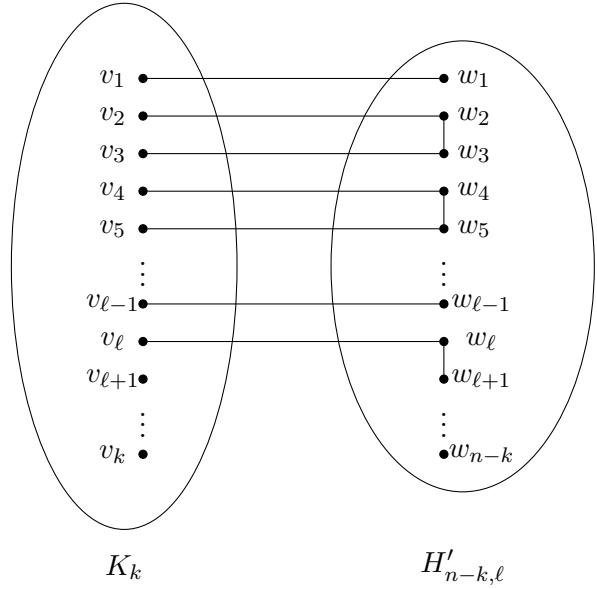
\begin{figure}[h]
\centering
\begin{tikzpicture}

\draw (-2.25,0) ellipse (1.5cm and 3.5cm);
\draw (-2.25,-4) node {$K_k$};

\node (v1) at (-2,2.5) [circle,draw,fill,scale=.3] {};
\node at (-2.4,2.5) {$v_1$};
\node (v2) at (-2,2) [circle,draw,fill,scale=.3] {};
\node at (-2.4,2) {$v_2$};
\node (v3) at (-2,1.5) [circle,draw,fill,scale=.3] {};
\node at (-2.4,1.5) {$v_3$};
\node (v4) at (-2,1) [circle,draw,fill,scale=.3] {};
\node at (-2.4,1) {$v_4$};
\node (v5) at (-2,0.5) [circle,draw,fill,scale=.3] {};
\node at (-2.4,0.5) {$v_5$};
\node (vdots1) at (-2,0) {$\vdots$};
\node (vellminus1) at (-2,-0.5) [circle,draw,fill,scale=.3] {};
\node at (-2.4,-0.5) {$v_{\ell-1}$};
\node (vell) at (-2,-1) [circle,draw,fill,scale=.3] {};
\node at (-2.4,-1) {$v_{\ell}$};
\node (vellplus1) at (-2,-1.5) [circle,draw,fill,scale=.3] {};
\node at (-2.4,-1.5) {$v_{\ell+1}$};
\node (vdots2) at (-2,-2) {$\vdots$};
\node (vk) at (-2,-2.5) [circle,draw,fill,scale=.3] {};
\node at (-2.4,-2.5) {$v_{k}$};

\draw (2.25,0) ellipse (1.75cm and 3cm);
\draw (2.25,-4) node {$H'_{n-k,\ell}$};

\node (w1) at (2,2.5) [circle,draw,fill,scale=.3] {};
\node at (2.4,2.5) {$w_1$};
\node (w2) at (2,2) [circle,draw,fill,scale=.3] {};
\node at (2.4,2) {$w_2$};
\node (w3) at (2,1.5) [circle,draw,fill,scale=.3] {};
\node at (2.4,1.5) {$w_3$};
\node (w4) at (2,1) [circle,draw,fill,scale=.3] {};
\node at (2.4,1) {$w_4$};
\node (w5) at (2,0.5) [circle,draw,fill,scale=.3] {};
\node at (2.4,0.5) {$w_5$};
\node (wdots1) at (2,0) {$\vdots$};
\node (wellminus1) at (2,-0.5) [circle,draw,fill,scale=.3] {};
\node at (2.5,-0.5) {$w_{\ell-1}$};
\node (well) at (2,-1) [circle,draw,fill,scale=.3] {};
\node at (2.5,-1) {$w_{\ell}$};
\node (wellplus1) at (2,-1.5) [circle,draw,fill,scale=.3] {};
\node at (2.5,-1.5) {$w_{\ell+1}$};
\node (wdots2) at (2,-2) {$\vdots$};
\node (wnminusk) at (2,-2.5) [circle,draw,fill,scale=.3] {};
\node at (2.5,-2.5) {$w_{n-k}$};

\foreach \from/\to in {v1/w1,v2/w2,v3/w3,v4/w4,v5/w5,vellminus1/wellminus1, vell/well, w2/w3, w4/w5, well/wellplus1}
		\draw (\from) -- (\to);

\end{tikzpicture}
\caption{The construction for $G^*$ when $\ell < k-1$ and $\ell \le n-k$.}
\label{fig-}
\end{figure}

Now $G^*$ is $k$-chromatic since the subgraph $K_k$ requires $k$ colors. 
We show the graph is also $\ell$-connected since any two vertices $v$ and $w$ are connected by $\ell$ disjoint paths. This is true if $v$ and $w$ both belong to the subgraph $K_k$ since $k>\ell$. %We claim this is also true if $v$ and $w$ belong to $H'_{n-k,\ell}$.

Suppose $v$ and $w$ both belong to  $H'_{n-k,\ell}$. By construction, the Harary graph $H_{n-k,\ell}$ is $\ell$-connected, so there are $\ell$ disjoint paths between any two vertices. We extend these disjoint paths to $H'_{n-k,\ell}$ over deleted edges, $w_iw_{i+1}$, by instead using edges $w_iv_i$, $v_iv_{i+1}$, and $v_{i+1}w_{i+1}$ in $G^*$ (with the obvious modification if $\ell$ and $n-k$ are odd and the degree $\ell+1$ vertex is internal on one of the paths). This covers all cases except when $v$ itself has degree $\ell+1$, in which case $v=w_j$ for some $j$.  But in this case since 
%
%In the case where both $\ell$ and $n-k$ are odd, $H_{n-k, \ell}$ has one higher degree vertex $w_j$. The path extension above does not preserve disjointness of two paths each using one of the deleted edges $w_{j-1}w_j$ and $w_{j}w_{j+1}$. Note that since 
$v$ has degree $\ell+1$, we can find $\ell$ disjoint paths between $v$ and any other vertex that excludes one of the edges $w_{j-1}v$ or $vw_{j+1}$; these $\ell$ disjoint paths can be extended to $H'_{n-k,j}$ as above.

Lastly, suppose $v$ belongs to the $K_k$ subgraph and $w$ belongs to the $H'_{n-k,\ell}$ subgraph. From $v$ there are $\ell$ disjoint paths to $H'_{n-k,\ell}$ each ending at one of $w_1$, $w_2$, \ldots, $w_\ell$. Since $H_{n-k,\ell}$ is $\ell$-connected, there exists $\ell$ disjoint paths from the subset of vertices $w_1$, $w_2$, \ldots, $w_\ell$ to the vertex $w$, and the same $\ell$ paths exists in $H'_{n-k,\ell}$. Therefore in all cases there are $\ell$ disjoint paths between the vertices $v$ and $w$.  Since deleting the endpoints of the matching in $K_k$ disconnects the graph, this shows that $G^*$ is $\ell$-connected.

\begin{proof}[Proof of Theorem \ref{thm-minedges1}]
Sharpness follows from the graph $G^*$. To prove the lower bound, let $G$ be a $k$-chromatic $\ell$-connected graph. We will consider two cases: if $G$ is $k$-critical and otherwise.

If $G$ is $k$-critical, then all vertices have degree at least $k-1$, so $G$ has at least $\frac{n(k-1)}{2}$ edges. Then the difference in the number of edges between $G$ and $G^*$ is at least
\begin{align*}
\frac{n(k-1)}{2}-\binom{k}{2} - \frac{(n-k+1)\ell}{2} &=\frac{1}{2}\left(n(k-1)-k(k-1)-(n-k+1)\ell\right)\\
&=\frac{1}{2}((k-1)(n-k)-(n-k)\ell-\ell)\\
&=\frac{1}{2}((n-k)(k-1-\ell) -\ell)
\end{align*}

Since $k-1>\ell$, we have $k-1-\ell\geq 1$, and combining this with $(n-k)\geq \ell$ gives $(n-k)(k-1-\ell)\geq \ell$.  Thus, 

\[\frac{1}{2}((n-k)(k-1-\ell) -\ell)\geq 0\]
and so the bound is correct in this case. 

Suppose $G$ is not $k$-critical.  Then $G$ has an (induced) $k$-critical subgraph.  This subgraph, $H$, has at least $k$ vertices and minimum degree at least $k-1$. Say $H$ has $k\leq x \leq n-1$ vertices.  Then $H$ has at least $\frac{x(k-1)}{2}$ edges.  Consider the vertices in $V(G)\setminus V(H)$.  Since $G$ is $\ell$-connected, there must be at least $\ell$ disjoint paths between $V(G)$ and $V(G)\setminus V(H)$.  This requires at least $\ell$ edges; assume there are $p \geq \ell$ edges with one endpoint in $V(H)$ and the other endpoint in $V(G) \setminus V(H)$.  Moreover, every vertex in $V(G)\setminus V(H)$ has to have minimum degree at least $\ell$, since $G$ is $\ell$-connected.  This requires a minimum of an additional 
\[ \dfrac{(n-x)\ell-p}{2}\]
edges with both endpoints in $V(G)\setminus V(H)$. 

In total, $G$ must have at least 
\[ \frac{x(k-1)}{2} + p + \dfrac{(n-x)\ell-p}{2} \geq \frac{x(k-1)}{2} + \ell + \dfrac{(n-x)\ell-\ell}{2}\]
edges. This bound is linear in $x$ with positive slope $\frac{k-1-\ell}{2}$, and so is minimized by the minimum value of $x$. Since $x\ge k$, we get 
\[ \frac{x(k-1)}{2} + \ell + \dfrac{(n-x)\ell-\ell}{2} \ge \frac{k(k-1)}{2} + \ell + \dfrac{(n-k)\ell-\ell}{2}\]
proving the claimed bound in this case. This finishes the proof.
\end{proof}

\subsection{Proof of Theorem \ref{thm-minedges2}}

We again start by constructing a graph that achieves the bound. Consider the disjoint union of the complete graph $K_k$ and the complete graph $K_{n-k}$.  Fix $\ell$ vertices $v_{1},v_{2},\ldots,v_{\ell} \in V(K_k)$, and label the vertices in $K_{n-k}$ by $w_1,\ldots,w_{n-k}$.  For a fixed $i$, add the $\ell-(n-k-1)$ edges joining $w_i$ and $v_j$ for each $j$ satisfying $i \leq j \leq i+(\ell-n+k)$. %\textcolor{blue}{Where indices are taken mod $k$ if I'm thinking about this right? JK - the highest $i$ is $n-k$ and $(n-k)+(\ell-n+k) = \ell$. So we will only end up connecting to the $\ell$ vertices $v_1,...,v_\ell$ anyway. - LK}  
Call this graph $G^*$, and note that $G^*$ has $\binom{k}{2} + \binom{n-k}{2} + (n-k)(\ell-(n-k-1))$ edges. See Figure \ref{fig-G*}.

\begin{comment}

\textcolor{red}{JE: The motivation here is that we're trying to do a matching between vertices in $K_k$ and $K_{n-k}$ (like in the previous case; note that $K_{n-k}$ is a Harary graph!), but a matching isn't going to give the correct minimum degree condition, since $K_{n-k}$ vertices have ``small'' degree.  So do a ``star-matching'' (my terminology?) with $\ell$ vertices from $K_k$:  A matching uses $K_{1,1}$'s; here we're using $K_{1,\ell-(n-k-1)}$'s to get the correct minimum degree, and done systematically so the overlap of stars is straightforward and nothing degenerate happens.  That was at least my thought here.}
\end{comment}

We first claim that $G^*$ is $k$-chromatic.  It requires $k$ colors on $K_k$. And each vertex $w_i$ can be colored with the color on $v_{i-1}$ (where we consider $v_0$ to be the vertex $v_{\ell}$).  This is a $k$-coloring of $G^*$. %\textcolor{blue}{As long as $\ell-(n-k-1)<k$, or $n>\ell+1$. Which it is, since $n>\ell+k$ and $k=1$ is silly (I think) -LK. Actually later you say $k\geq 4$, so this is fine.}

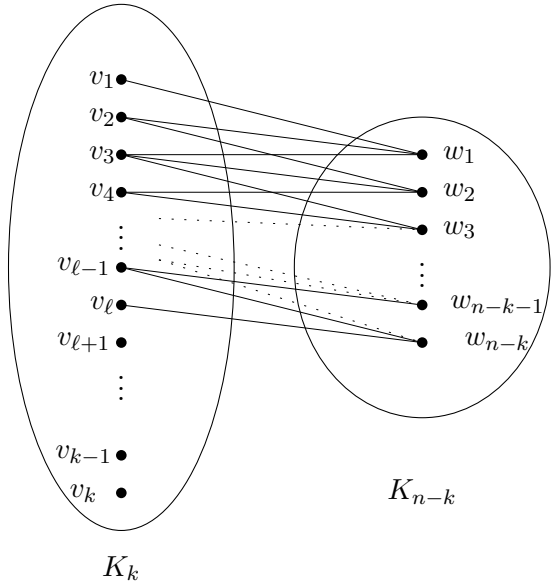
\begin{figure}[ht!]
\begin{center}

\begin{tikzpicture}
\draw (-2,0) ellipse (1.5cm and 3.5cm);
\draw (2,0) ellipse (1.7cm and 2cm);

\coordinate(v1) at (-2,2.5);
\coordinate(v2) at (-2,2);
\coordinate(v3) at (-2,1.5);
\coordinate(v4) at (-2,1);
\coordinate(vdots) at (-2,.5);
\coordinate(vellminus1) at (-2,0);
\coordinate(vell) at (-2,-.5);
\coordinate(vellplus1) at (-2,-1);
\coordinate(velldots) at (-2,-1.5);
\coordinate(vkminus1) at (-2,-2.5);
\coordinate(vk) at (-2,-3);

\coordinate(w1) at (2,1.5);
\coordinate(w2) at (2,1);
\coordinate(w3) at (2,.5);
\coordinate(wdots) at (2,0);
\coordinate(wnminuskminus1) at (2,-.5);
\coordinate(wnminusk) at (2,-1);

\fill (v1) circle (2pt);
\draw (-2.25,2.5) node {$v_1$};
\fill (v2) circle (2pt);
\draw (-2.25,2) node {$v_2$};
\fill (v3) circle (2pt);
\draw (-2.25,1.5) node {$v_3$};
\fill (v4) circle (2pt);
\draw (-2.25,1) node {$v_4$};
\draw (vdots) node {$\vdots$};
\fill (vellminus1) circle (2pt);
\draw (-2.5,0) node {$v_{\ell-1}$};
\fill (vell) circle (2pt);
\draw (-2.25,-.5) node {$v_{\ell}$};
\fill (vellplus1) circle (2pt);
\draw (-2.5,-1) node {$v_{\ell+1}$};
\draw (velldots) node {$\vdots$};
\fill (vkminus1) circle (2pt);
\draw (-2.5,-2.5) node {$v_{k-1}$};
\fill (vk) circle (2pt);
\draw (-2.5,-3) node {$v_{k}$};
\fill (w1) circle (2pt);
\draw (2.5,1.5) node {$w_{1}$};
\fill (w2) circle (2pt);
\draw (2.5,1) node {$w_{2}$};
\fill (w3) circle (2pt);
\draw (2.5,.5) node {$w_{3}$};
\draw (wdots) node {$\vdots$};
\fill (wnminuskminus1) circle (2pt);
\draw (3,-.5) node {$w_{n-k-1}$};
\fill (wnminusk) circle (2pt);
\draw (3,-1) node {$w_{n-k}$};

\tikzstyle{EdgeStyle}=[-,ultra thin]
	\Edge(v1)(w1);
    \Edge(v2)(w1);
    \Edge(v3)(w1);
    \Edge(v2)(w2);
    \Edge(v3)(w2);
    \Edge(v4)(w2);
    \Edge(v3)(w3);
    \Edge(v4)(w3);

    \Edge(wnminuskminus1)(vellminus1);
    \Edge(wnminusk)(vellminus1);
    \Edge(wnminusk)(vell);

\tikzstyle{EdgeStyle}=[loosely dotted,ultra thin]
	\Edge(-1.5,.65)(w3);
    \Edge(-1.5,.1)(wnminuskminus1);
    \Edge(-1.5,.3)(wnminuskminus1);
    \Edge(-1.5,.1)(wnminusk);

\draw (-2,-4) node {$K_k$};
\draw (2, -3) node {$K_{n-k}$};

\end{tikzpicture}
\end{center}

\caption{The graph $G^*$ when $\ell-(n-k-1)=3$; note that here $w_2$ is adjacent to $v_2$, $v_3$, and $v_4$.}
\label{fig-G*}
\end{figure}

Next, we claim that $G^*$ is $\ell$-connected; note that removing $v_{1},\ldots,v_{\ell}$ disconnects the graph (as $\ell<k$).  We claim that any two vertices $v$ and $w$ are connected by $\ell$ disjoint paths.  This is clear if $v$ and $w$ are both in $K_k$, since $\ell<k-1$. It is also clear if $v$ and $w$ are both in $K_{n-k}$, as there are $n-k-1$ paths in $K_{n-k}$, and the $\ell-(n-k-1)$ edges to $K_k$ from each vertex lead to $\ell-(n-k-1)$ other edge disjoint paths.  So assume $v$ is some vertex in $K_k$ and $w=w_1$.  We know that $w$ has neighbors $v_i$ for $1\leq i \leq \ell-n+k+1$; those have disjoint paths (with $0$ or $1$ edge) to $v$.  And furthermore the neighbors $w_m$, $m >1$, can use their edge to $v_{m+\ell-n+k}$ and then edge $v_{m+\ell-n+k}v$ to produce the remaining disjoint paths from $w$ to $v$.

\begin{proof}[Proof of Theorem \ref{thm-minedges2}]
Note that we must have $k \geq 4$, since $k \leq 3$ implies $\ell<2$, which contradicts the assumption of $\ell>1$. %, since in this case we have $\ell<2$, and the condition $\ell > n-k$ means $n<3+\ell \leq 4$.  So assume $k \geq 4$ in this section.  
Also note that the inequalities together imply that $n<2k-1$. Furthermore, $n=k$ is not possible, as then $G=K_k$ which is $k-1$ connected, but $\ell=k-1$ is not allowed.  Therefore we can assume $n>k$.

Sharpness comes from the graph $G^*$. Suppose $G$ is a $k$-chromatic $\ell$-connected graph with $\ell>n-k$ and $k-1>\ell>1$.  We use a result from Gallai \cite{Gallai}, which says that if $k \geq 4$ and $k+2 \leq n \leq 2k-1$, a $k$-critical graph on $n$ vertices satisfies
\[
|E(G)| \geq \frac{1}{2}\left( n(k-1) + (n-k)(2k-n)-2\right);
\]
there is a characterization of equality given in that paper as well.  It is noted in various places that no $k$-critical graph has $n=k+1$ vertices.
We consider two cases: $G$ is $k$-critical and otherwise.

If $G$ is $k$-critical, then since $k\geq 4$ and our ranges of $n$, $k$, and $\ell$ imply that $n<2k-1$, the result from Gallai \cite{Gallai} gives at least $\frac{1}{2}( n(k-1) + (n-k)(2k-n)-2)$ edges. Then the difference between $|E(G)|$ and $|E(G^*)|$ is at least

\begin{align*}
\frac{1}{2}&\left( n(k-1) + (n-k)(2k-n)-2\right) - \binom{k}{2} - \binom{n-k}{2} - (n-k)(\ell-(n-k-1))\\
&=\frac{n-k}{2}\left( (k-1) + (2k-n)-(n-k-1)-2\ell+2(n-k-1) \right) - 1\\
&= \frac{n-k}{2}\left( (k-1) + (2k-n)+(n-k-1)-2\ell \right) -1\\
&= \frac{n-k}{2}\left( (k-1) -\ell + (k-1)-\ell \right) -1\\
&= (n-k)(k-1-\ell) - 1.
%&=\frac{(n-k)}{2}\left(n(k-1)+(n-k)(2k-n)-2)-k(k-1)-(n-k)(n-k-1)-(n-k)(\ell-(n-k-1)) \right)-1
\end{align*}
%\textcolor{teal}{In the first line of the equalities above, I believe the term $2(n-k)(\ell-(n-k-1)$ should be $(n-k)(\ell-(n-k-1)$. Otherwise, the algebra checks out on my end. -TS} \textcolor{red}{JE: changed. Good catch. } 
Now $k-1-\ell>0$ and $n-k\geq 2$ (as there is no $k$-critical graph on $n=k+1$ vertices), and so the above expression is positive, which implies that bound is correct in this case.

Suppose now that $G$ is not $k$-critical.  Then $G$ again has an induced $k$-critical subgraph $H$ with $k \leq x \leq n-1$ vertices.  Since $n<2k-1$, if $x>k+1$ then the Gallai bound applies, so $H$ has at least $\frac{1}{2}( x(k-1) + (x-k)(2k-x)-2)$ edges. If instead $x=k$, the graph $H$ has $\frac{1}{2}k(k-1)$ edges. We cannot have $x=k+1$ as there is no $k$-critical graph on $x=k+1$ vertices.

Consider the vertices in $V(G) \setminus V(H)$. Each vertex must have at least $\ell$ disjoint paths to $V(H)$, and so in particular must have minimum degree at least $\ell$.  Therefore the degree sum of the vertices in $V(G) \setminus V(H)$ is at least $(n-x)\ell$.  There can be at most $\binom{n-x}{2}$ edges that contribute two to this degree sum, coming from edges with both endpoints in $V(G) \setminus V(H)$. This means there must be at least $(n-x)\ell - (n-x)(n-x-1) = (n-x)(\ell-(n-x-1))$ edges between $H$ and $G-H$.  And if there are $p$ edges missing inside the induced subgraph on $V(G) \setminus V(H)$, where $0 \leq p$, then we have at least $(n-x)(\ell-(n-x-1))+2p$ edges between $H$ and $G-H$.

Therefore the total number of edges in $G$ is at least
\[
\frac{( x(k-1) + (x-k)(2k-x)-2 \cdot \textbf{1}_{\{x\geq k+2\}})}{2} + \frac{(n-x)(n-x-1)}{2} - p + (n-x)(\ell-(n-x-1))+2p,
\]
where $\textbf{1}_{\{x\geq k+2\}}$ is the indicator function on the event $x \geq k+2$.  Minimizing this, we take $p=0$, giving at least
\[
\frac{( x(k-1) + (x-k)(2k-x)-2\cdot \textbf{1}_{\{x\geq k+2\}})}{2} + \frac{(n-x)(n-x-1)}{2} + (n-x)(\ell-(n-x-1))
\]
edges. The expression
\[
\frac{x(k-1) + (x-k)(2k-x)-2\cdot \textbf{1}_{\{x\geq k+2\}} + (n-x)(n-x-1) + 2(n-x)(\ell-(n-x-1))}{2}
\]
is a quadratic function of (real-valued) $x$ with leading coefficient $-1$.  By the Extreme Value Theorem, the minimum occurs over $k+2 \leq x \leq n-1$ at $x=k+2$ or $x=n-1$. We also need to compare these values to $x=k$, the other possible value for $x$. (We remark that we separate out the case when $x=k$ as the indicator function is not continuous, and so we cannot apply the Extreme Value Theorem on $k \leq x \leq n-1$.) When $x=k$ we have the bound
\[
\frac{k(k-1) + (n-k)(n-k-1) + 2(n-k)(\ell-(n-k-1))}{2}.
\]

We need to show that when $x=k+2$ or $x=n-1$, we have a larger bound. For $x=k+2$, the bound is
\[
\frac{(k+2)(k-1) + 2(k-2)-2 + (n-k-2)(n-k-3) + 2(n-k-2)(\ell-(n-k-3))}{2}
\]
and when $x=n-1$ the bound is
\[
\frac{(n-1)(k-1) + (n-1-k)(2k-n+1)-2  + 2\ell}{2}.
\]
Computing the $x=k+2$ count minus the $x=k$ count gives
\[
2(n-\ell)-7.
\]
Now for this bound we have $n \geq k+2$, so $\ell<k-1$ means $\ell < n-3$, or $4 \leq n-\ell$, so the difference in terms is positive in this case.

Computing the $x=n-1$ count minus the $x=k$ count gives
\[
(k-\ell)(n-1-k)-1;
\]
here $k-\ell\geq 2$ and $n-1>k$ (as $n \geq k+2$). This shows the difference is positive in this case.  Therefore, the minimum value occurs for $x=k$, proving the claimed bound.
\end{proof}

\section{Concluding Remarks}\label{sec-conclusion}
In this section we highlight a few open problems that are related to the contents of this paper.  In Theorem \ref{thm-asympresult}, we characterized the $n$-vertex $k$-chromatic $\ell$-connected graph with the maximum number of independent sets for large $n$. We expect the result to hold for all $n$ for which the graph $G^*$ is $k$-chromatic and $\ell$-connected.
\begin{conjecture}\label{conj-1}
Let $3 \leq k \leq \ell$ and $n \geq 2\ell$. If $G$ is an $n$-vertex $k$-chromatic $\ell$-connected graph, then 
\[
i(G) \leq i(G^*).
\]
\end{conjecture}

\begin{conjecture}\label{conj-2}
Let $2 \leq \ell < k$ and $n\neq 5$. If $G$ is an $n$-vertex $k$-chromatic $\ell$-connected graph, then
\[
i(G) \leq i(G^*).
\]
\end{conjecture}
Conjecture \ref{conj-2} is true for $k=3$ and $\ell=2$ (for $n \neq 5$), as we showed in Section \ref{sec-2con3chrom}.
%\begin{conjecture}
%Let $k\geq 3$, $\ell \geq 2$ be fixed.
%\begin{enumerate}
%    \item If $k \leq \ell$, $n \geq 2\ell$, and $G$ is an $n$-vertex $k$-chromatic $\ell$-connected graph, then $i(G) \leq i(G^*)$.
%    \item If $k > \ell$, $n>k$, and $G$ is an $n$-vertex $k$-chromatic $\ell$-connected graph, then $i(G) \leq i(G^*)$.
%\end{enumerate}
%\end{conjecture}
%
There are also open questions related to the number of independent sets of size $t$.  We expect the following to hold, which extends Theorem \ref{thm-bigt} down to $t \geq 3$.
\begin{conjecture}\label{conj-3}
Let $3 \leq k \leq \ell$ and $n \geq 2\ell$. If $G$ is an $n$-vertex $k$-chromatic $\ell$-connected graph and $t \geq 3$, then
\[
i_t(G) \leq i_t(G^*).
\]
\end{conjecture}
It is also natural to conjecture that this behavior holds for $k>\ell$ as well.
\begin{conjecture}\label{conj-4}
Let $k \geq 4$ and $k > \ell$. If $G$ is an $n$-vertex $k$-chromatic $\ell$-connected graph and $t \geq 3$, then
\[
i_t(G) \leq i_t(G^*).
\]
\end{conjecture}
These last two conjectures appeared as questions in \cite{EngbersErey}. We note that the corresponding results for $k=3$ and $\ell=2$ are shown in Theorem \ref{thm-2con3chrom}, and that the cases when $\ell=0$ and $\ell=1$ appear in \cite{EngbersErey}.

\end{document}